\documentclass[12pt]{amsart}
\usepackage{a4wide}
\usepackage{amssymb}
\usepackage{amsthm}
\usepackage{amsmath}
\usepackage{amscd}
\usepackage{verbatim}
\usepackage{color}
\usepackage{hyperref}
\textheight8.5in \textwidth6.75in \numberwithin{equation}{section}

\theoremstyle{plain}
\newtheorem{theorem}{Theorem}[section]
\newtheorem{corollary}[theorem]{Corollary}
\newtheorem{lemma}[theorem]{Lemma}
\newtheorem{proposition}[theorem]{Proposition}

\theoremstyle{definition}
\newtheorem{definition}[theorem]{Definition}
\newtheorem{example}{Example}

\theoremstyle{remark}
\newtheorem*{remark}{Remark}

\renewcommand{\Re}{\operatorname{Re}}
\renewcommand{\Im}{\operatorname{Im}}
\newcommand{\Dhat}{\widehat{D}}
\newcommand{\R}{\mathbb{R}}
\newcommand{\Q}{\mathbb{Q}}
\newcommand{\Z}{\mathbb{Z}}
\newcommand{\N}{\mathbb{N}}
\newcommand{\C}{\mathbb{C}}

\renewcommand{\H}{\mathbb{H}}
\renewcommand{\L}{\mathbb{L}}

\renewcommand{\P}{\mathbb{P}}

\newcommand{\kzxz}[4]{\left(\begin{smallmatrix} #1 & #2 \\ #3 & #4\end{smallmatrix}\right) }
\newcommand{\kabcd}{\kzxz{a}{b}{c}{d}}

\newcommand{\calF}{\mathcal{F}}

\newcommand{\calM}{\mathcal{M}}

\newcommand{\calV}{\mathcal{V}}

\newcommand{\eps}{\varepsilon}

\newcommand{\Orth}{\operatorname{O}}

\newcommand{\SL}{{\text {\rm SL}}}

\newcommand{\smallabcd}{\left(\begin{smallmatrix}a & b \\ c & d\end{smallmatrix}\right)}

\newcommand{\pihol}{\pi_{hol}}

\begin{document}

\title[Special values of shifted convolution Dirichlet series]{Special values of shifted convolution Dirichlet series }

\dedicatory{For Jeff Hoffstein on his 61st birthday.}

\author{Michael H. Mertens and  Ken Ono}

\address{Mathematisches Institut der Universit\"at zu K\"oln, Weyertal 86-90,
D-50931 K\"oln, Germany} \email{mmertens@math.uni-koeln.de}

\address{Department of Mathematics and Computer Science, Emory University,
Atlanta, Georgia 30022} \email{ono@mathcs.emory.edu}
%\thanks{}


\begin{abstract}
In a recent important paper, Hoffstein and Hulse \cite{HoffsteinHulse} generalized the notion of
Rankin-Selberg convolution $L$-functions by defining
{\it shifted convolution} $L$-functions.
We investigate symmetrized versions of their functions, and
 we prove that the generating functions of certain special values
are linear combinations of weakly holomorphic quasimodular forms and ``mixed mock modular'' forms.
\end{abstract}

\maketitle

\section{Introduction and Statement of Results}\label{sect:intro}

Suppose that $f_1(\tau)\in S_{k_1}(\Gamma_0(N))$ and
$f_2(\tau)\in S_{k_2}(\Gamma_0(N))$
 are cusp forms of even weights $k_1\geq k_2$ with $L$-functions
$$
L(f_i,s)=\sum_{n=1}^{\infty}\frac{a_i(n)}{n^s}.
$$
In the case of equal weights, Rankin and Selberg \cite{Rankin, Selberg} independently introduced the so-called {\it Rankin-Selberg convolution}
$$
L(f_1\otimes f_2,s):=\sum_{n=1}^{\infty}\frac{a_{1}(n)\overline{a_{2}(n)}}{n^s},
$$
works which are among the most important contributions to the modern theory of automorphic forms.
Later in 1965, Selberg \cite{Selberg2} introduced {\it shifted convolution} $L$-functions, and these
series have now played an important role in progress towards Ramanujan-type conjectures for Fourier coefficients and the
Lindel\"of Hypothesis for
automorphic $L$-functions inside the critical strip (for example, see \cite{Blomer, Good1, Good2, Survey} and the references therein).

In a recent paper, Hoffstein and Hulse
 \cite{HoffsteinHulse}\footnote{Here we choose slightly different normalizations
for Dirichlet series from those that appear in \cite{HoffsteinHulse}.} introduced the
 {\it shifted convolution
series}
\begin{equation}\label{shiftedseries}
D(f_1,f_2,h;s):=\sum_{n=1}^{\infty}\frac{a_{1}(n+h)\overline{a_{2}(n)}}{n^s}.
\end{equation}
When $k_1=k_2$ they obtained the meromorphic continuation of these series and certain multiple Dirichlet series
which are obtained by additionally summing in $h$ aspect. Moreover, as an important application they
obtain a Burgess-type bound for $L$-series associated to modular forms.

Here we study the arithmetic properties of these Dirichlet series, where we additionally allow the weights to be non-equal. For this, it is convenient to consider the
{\it derived} shifted convolution series
\begin{equation}\label{derivedshiftedseries}
D^{(\mu)}(f_1, f_2,h;s):=
\sum_{n=1}^{\infty}\frac{a_{1}(n+h)\overline{a_{2}(n)}(n+h)^{\mu}}{n^s},
\end{equation}
where $h\geq 1$ and $\mu\geq 0$ are integers 
and $\Re(s)> \tfrac{k_1+k_2}{2}+\mu$. Of course, we have that $D^{(0)}(f_1,f_2,h;s)=D(f_1,f_2,h;s)$.

For each $\nu\geq 0$ and each $h\geq 1$ we define
(see Section~\ref{Lfcndef}) a {\it symmetrized shifted convolution} Dirichlet
series $\widehat{D}^{(\nu)}(f_1,f_2,h;s)$ using the $D^{(\mu)}(f_1,f_2,h;s)$.
We consider their special values at $s=k_1-1$ in $h$-aspect, which are well-defined as (conditionally) convergent series if $\nu\leq\tfrac{k_1-k_2}2$. They are absolutely convergent if $\nu<\tfrac{k_1-k_2}2$. In order to investigate these special values we
construct the $q$-series
\begin{equation}
\L^{(\nu)}(f_1,f_2;\tau):=\sum_{h=1}^{\infty}\widehat{D}^{(\nu)}(f_1,f_2,h;k_1-1)q^h,
\end{equation}
where $q:=e^{2\pi i \tau}$. 
In the special case where $\nu=0$ and $k_1=k_2$, we have that
$$\Dhat(f_1,f_2,h;s)=D(f_1,f_2,h;s)-D(\overline{f_2},\overline{f_1},-h;s),
$$
 which then implies that
$$
\L^{(0)}(f_1,f_2;\tau)=\sum_{h=1}^{\infty}\Dhat(f_1,f_2,h;k_1-1)q^h.
$$

It is natural to ask for a characterization of these functions. For example,
consider the case where
$f_1=f_2=\Delta$, the unique normalized weight 12 cusp form on $\SL_2(\Z)$. Then we have that
$$
\L^{(0)}(\Delta,\Delta;\tau)=-33.383\dots q + 266.439\dots q^2 - 1519.218\dots q^3+ 4827.434\dots q^4-\dots
$$
Using the usual Eisenstein series $E_{2k}=E_{2k}(\tau)$ and Klein's $j$-function, we let
$$
\sum_{n=-1}^{\infty} r(n)q^n:=-\Delta(j^2-1464j-\alpha^2+1464\alpha),
$$
where $\alpha=106.10455\dots$. By letting $\beta=2.8402\dots$, we find that
\begin{displaymath}
\begin{split}
-\frac{\Delta}{\beta}&\left (\frac{65520}{691}+\frac{E_2}{\Delta}-\sum_{n\neq 0} r(n)n^{-11}q^n\right)\\
&=-33.383\dots q + 266.439\dots q^2 - 1519.218\dots q^3+ 4827.434\dots q^4-\dots.
\end{split}
\end{displaymath}
It turns out that this $q$-series indeed equals $\L^{(0)}(\Delta,\Delta;\tau)$, and the purpose of this
paper is to explain such formulae.

The $q$-series $\L^{(\nu)}(f_1,f_2;\tau)$ for $\nu=\tfrac{k_1-k_2}{2}$ arise from the {\it holomorphic projections} of completed
{\it mock modular forms}.
In 1980, Sturm \cite{St80} introduced the method of holomorphic projection, and in their celebrated work
 on Heegner points and derivatives of $L$-functions, Gross and Zagier \cite{GrZ86} further developed this technique.
  In unpublished work, Zwegers suggested applying
such holomorphic projections to the theory of mock modular forms.
Recently, Imamo\u{g}lu, Raum, and Richter \cite{IRR13} have obtained general theorems in this direction
with applications to Ramanujan's mock theta functions in mind. This work, as well as 
the preprint \cite{Mertpreprint} by the first author,  are based on unpublished notes of Zagier.
Here we continue this theme, and we show that the $\L^{(\nu)}(f_1,f_2;\tau)$, again for $\nu=\tfrac{k_1-k_2}{2}$, also arise in this way.

The results we obtain depend on the theory of {\it harmonic Maass forms}, certain nonholomorphic modular forms which explain
Ramanujan's enigmatic {\it mock theta functions} (see \cite{Ono08, ZagierBourbaki, ZwegersDiss} and the references therein).
For any $f_1$, we denote by $M_{f_1}$ a harmonic Maass form of weight $2-k_1$ whose \emph{shadow} is $f_1$ (see Section ~\ref{nutsandbolts} for the definition).
Such forms always exist.
Our main result gives a surprising relation between the $\nu$th \emph{Rankin-Cohen bracket} (again, see Section~\ref{nutsandbolts} for definitions) of $M_{f_1}$ and $f_2$ and the generating function $\L^{(\nu)}(f_1,f_2;\tau)$. These Rankin-Cohen brackets are (completed) \emph{mixed mock modular forms} (see Section~\ref{nutsandbolts}).

As a technical condition for our result, we need the notion of $M_{f_1}$ being \emph{good} for $f_2$. By this we mean that the Rankin-Cohen bracket function $[M_{f_1}^+,f_2]_\nu$, where $M_{f_1}^+$ denotes the \emph{holomorphic part} of $M_{f_1}$, is holomorphic on the upper-half of the complex plane, and is bounded as one approaches all representatives for all cusps (see Section~\ref{proofs} for details). Our main result relates such Rankin-Cohen bracket functions
to $\L^{(\nu)}(f_1,f_2;\tau)$ modulo $\widetilde{M}_{2}(\Gamma_0(N))$.
Here the space $\widetilde{M}_2(\Gamma_0(N))$ is given by
\begin{equation}
\widetilde{M}_2(\Gamma_0(N))=  \C E_2 \oplus M_2(\Gamma_0(N)
\end{equation}
where $M_k(\Gamma_0(N))$ is the space of weight $k$ holomorphic modular forms on $\Gamma_0(N)$.
Special care is required for weight 2 because the
Eisenstein series $E_2$ is not a holomorphic modular form; it is a {\it quasimodular form}.
In general we require the spaces
$\widetilde{M}^{!}_2(\Gamma_0(N))$, which are the extension of $\widetilde{M}_2(\Gamma_0(N))$ by the weight $2$ {\it weakly holomorphic
modular forms} on $\Gamma_0(N)$. Weakly holomorphic modular forms are those meromorphic modular forms whose poles (if any)
are supported at cusps. The space of such forms of weight $k$ on $\Gamma_0(N)$ is denoted by $M_k^{!}(\Gamma_0(N))$.

\begin{theorem}\label{mainthm}
Assume the notation above.  If $\nu= \tfrac{k_1-k_2}{2}$, then
\[\L^{(\nu)}(f_2,f_1;\tau) = -\frac{1}{(k_1-2)!}\cdot [M_{f_1}^{+},f_2]_{\nu}(\tau)+F(\tau),\]
where $F\in \widetilde{M}^{!}_{2}(\Gamma_0(N))$.
Moreover, if $M_{f_1}$ is good for $f_2$, then $F\in \widetilde{M}_{2}(\Gamma_0(N))$.
\end{theorem}

\medskip
\noindent{\it Three remarks.}

\smallskip
\noindent
(1) One can derive Theorem~\ref{mainthm} when $\nu=0$ and $M_{f_1}$ is good for $f_2$ from  Theorem~3.5 of \cite{IRR13}.

\smallskip
\noindent
(2) Given cusp forms $f_1$ and $f_2$, it is not generically true that there is a harmonic Maass form $M_{f_1}$ which is good for $f_2$.
To see this, consider the case where $f_1=f_2=f$, a normalized Hecke eigenform on $\SL_2(\Z)$. For $M_{f}$ to be good for $f$, it is necessary that
$M_{f}^{+}$ has at most a simple pole at infinity. By the theory of Poincar\'e series (see Section~\ref{poincareseries}) it is clear
that most Hecke eigenforms $f$ do not satisfy this condition.

\smallskip
\noindent
(3) The reason for the extra condition $\nu=\tfrac{k_1-k_2}{2}$ comes from the fact that on the one hand holomorphic projection only makes sense for weights at least $2$ (see Section \ref{sec:holproj}) and the non-holomorphic modular form $[M_{f_1},f_2]_{\nu}$ has weight $2\nu+2-k_1+k_2$, on the other hand we need $\nu\leq \tfrac{k_1-k_2}{2}$ to ensure convergence of the shifted convolution series $\widehat{D}^{(\nu)}(f_1,f_2,h;s)$ at $s=k_1-1$.
\medskip

The expression on the right hand side of Theorem~\ref{mainthm} is explicitly computable when $f_1=f_2$ by making use of the theory of Poincar\'e series
 (see Section~\ref{poincareseries} for notation and definitions).

\begin{corollary}\label{cuspidalgoodness} Suppose that $k\geq 2$ is even and $m$ is a positive integer.
If
$P(\tau):=P(m,k,N;\tau)\in S_{k}(\Gamma_0(N))$ and $Q(\tau):=
Q(-m,k,N;\tau)\in H_{2-k}(\Gamma_0(N))$, then
$$\L^{(0)}(P,P;\tau)= \frac{1}{m^{k-1}\cdot (k-1)!}\cdot Q^{+}(\tau) P(\tau) +F(\tau),
$$
where $F\in \widetilde{M}^{!}_2(\Gamma_0(N))$. Moreover, if $m=1$, then
$F\in \widetilde{M}_2(\Gamma_0(N))$.
\end{corollary}
\begin{remark}
In Section~\ref{poincareseries} we shall see that $Q(-m,k,N;\tau)$ has a pole of order $m$ at $i\infty$ while $P(m,k,N;\tau)$ generally only has a simple zero there. This explains the special role of $m=1$ in the corollary above.
\end{remark}

\begin{example} Here we consider the case where $f_1=f_2=\Delta$, the unique normalized
cusp form of weight 12 on $\SL_2(\Z)$. Using the first $10^7$ coefficients of $\Delta$, one obtains the following
numerical approximations for the first few shifted convolution
values $\Dhat(\Delta,\Delta,h;11)$:

\begin{displaymath}
\begin{array}{|c||c|c|c|c|c|}
\hline h & 1 & 2 & 3 & 4 & 5 \\
\hline \ & \ &   &   &   &   \\
\Dhat(\Delta,\Delta,h;11) & -33.383\dots & 266.439\dots & -1519.218\dots & 4827.434\dots & -5704.330\dots\\
\hline
\end{array}
\end{displaymath}

\medskip
\noindent
We have that
$\Delta(\tau)=\frac{1}{\beta}P(1,12,1;\tau)$, where $\beta$ can be described in terms of a Petersson norm, or
as an infinite sum of Kloosterman sums weighted by $J$-Bessel functions as follows
$$
\beta:=\frac{(4\pi)^{11}}{10!}\cdot \Vert P(1,12,1)\Vert^2  =1+2\pi \sum_{c=1}^{\infty}\frac{K(1,1,c)}{c}\cdot J_{11}(4\pi/c)=2.8402\dots.
$$
Since $m=1$ and there are no weight 2 holomorphic modular forms on $\SL_2(\Z)$, Corollary~\ref{cuspidalgoodness} then implies that
\begin{displaymath}
\begin{split}
\L^{(0)}(\Delta,\Delta;\tau)&=\frac{Q^{+}(-1,12,1;\tau)\cdot \Delta(\tau)}{11!\cdot \beta}-\frac{E_2(\tau)}{\beta}\\
&=-33.383\dots q +266.439\dots q^2-1519.218\dots q^3 + 4827.434\dots q^4-\dots.
\end{split}
\end{displaymath}
Since the expressions for the coefficients of $Q^{+}(-1,12,1;\tau)$ are rapidly convergent,
this $q$-series identity provides an extremely efficient method for computing the values $\Dhat(\Delta,\Delta,h; 11)$.
\end{example}

\begin{example} Here we consider the newform $f=f_1=f_2=\eta(3\tau)^8\in S_4(\Gamma_0(9))$.
This form has complex multiplication by $\Q(\sqrt{-3})$, and is a multiple of the Poincar\'e
 series $P(1,4,9;\tau)$ because this space of cusp forms is one-dimensional. Here we show how complex multiplication
implies some striking {\it rationality} properties for the corresponding special values
$\Dhat(f,f,h;3)$.
 Using a computer, one obtains the following
numerical approximations for the first few shifted convolution values.

\begin{displaymath}
\begin{array}{|c||c|c|c|c|c|}
\hline h & 3 & 6 & 9 & 12 & 15 \\
\hline \ & \ &   &   &   &   \\
\Dhat(f,f,h;3) & -10.7466\dots & 12.7931\dots & 6.4671\dots & -79.2777\dots & 64.2494\dots\\
\hline
\end{array}
\end{displaymath}

\medskip
\noindent
We note that $\Dhat(f,f,h;3)=0$ whenever $n$ is not a multiple of $3$.

Below we define real numbers $\beta, \gamma,$ and $\delta$ which are approximately
$$\beta:=\frac{(4\pi)^3}{2}\cdot \Vert P(1,4,9)\Vert^2=1.0468\dots,\ \ \  \gamma=-0.0796\dots,\ \ \ \delta=-0.8756\dots.
$$
As we shall see, these real numbers arise naturally from the theory of Petersson inner products.
In the table below we consider the first few values of
$$
   T(f;h):=\beta \Dhat(f,f,h;3)+24\beta \gamma\sum_{d\mid h} d- 12\beta \delta\sum_{\substack{d\mid h\\ 3\nmid d}} d.
$$

\begin{displaymath}
\begin{array}{|c||c|c|c|c|}
\hline h & 3 & 6 & 9 & 12  \\
\hline \ & \ &   &   &     \\
T(f; h) & -8.250...\sim -\frac{33}{4} & 22.391... \sim \frac{2799}{125} & -8.229...\sim
-\frac{32919}{4000}& -61.992...\sim -\frac{8250771}{133100}\\
 \ & & & & \\
\hline
\end{array}
\end{displaymath}
It is indeed true that $T(f;h)$ is rational for all $h$.
To see this, we note that
$f=\frac{1}{\beta}P(1,4,9;\tau)$.
By applying Corollary~\ref{cuspidalgoodness} with $m=1$, we obtain a congruence between
$Q^{+}(-1,4,9;\tau)f(\tau)/\beta$ and $\L^{(0)}(f,f;\tau)$ modulo $
\widetilde{M}_2(\Gamma_0(9))$. In particular, we obtain the following identity
$$
\L^{(0)}(f,f;\tau)- \frac{Q^{+}(-1,4,9;\tau)f(\tau)}{\beta}=\gamma\left(1-24\sum_{n=1}^{\infty}
\sigma_1(3n)q^{3n}\right) + \delta\left(1+12\sum_{n=1}^{\infty}\sum_{\substack{d\mid 3n\\ 3\nmid d}}dq^{3n}\right).
$$
The rationality of $T(f;h)$ follows from the deep fact that 
$$
Q^{+}(-1,4,9;\tau)=q^{-1}-\frac{1}{4}q^2+\frac{49}{125}q^5-\frac{3}{32}q^8-\dots,
$$
has rational Fourier coefficients. Bruinier, Rhoades and the second author proved a general theorem
(see Theorem 1.3 of \cite{BOR}) which establishes the algebraicity of mock modular forms
whose shadows are CM forms.
\end{example}

\begin{example}
Here we consider the case of the weight 24 Poincar\'e series
\begin{displaymath}
\begin{split}
f(\tau&):=P(2,24,1;\tau)\\
& =0.00001585\dots q+2.45743060\dots q^2- 114.85545780\dots q^3+ 2845.49507680\dots q^4-\dots.
\end{split}
\end{displaymath}
We have the following
numerical approximations for the first few shifted convolution
values $\Dhat(f,f,h;23)$:

\begin{displaymath}
\begin{array}{|c||c|c|c|c|c|}
\hline h & 1 & 2 & 3 & 4 & 5 \\
\hline \ & \ &   &   &   &   \\
\Dhat(f,f,h;23) & -0.00000629\dots & -0.00092041\dots & 0.033927\dots & -0.472079\dots & 4.628028\dots\\
\hline
\end{array}
\end{displaymath}
We define
$$F:=\frac{1}{2^{23}}\cdot ( \gamma\cdot \frac{E_4^2 E_6}{\Delta}+\delta \cdot E_2)=\frac{\gamma}{2^{23}}\cdot q^{-1}+\dots
$$
where $\gamma=-0.00001585\dots$ and $\delta=-2.45743\dots$.  We have that $F\in \widetilde{M}^{!}_2(\Gamma_0(1))\setminus
\widetilde{M}_2(\Gamma_0(1))$. It turns out that
\begin{displaymath}
\begin{split}
&\L^{(0)}(f,f;\tau)=\frac{Q^{+}(-2,24,1;\tau)\cdot P(2,24,1;\tau)}{2^{23}\cdot 23!}+F\\
&\ \ \ =-0.00000629\dots q-0.000920\dots q^2+0.0339\dots q^3-0.4720\dots q^4+4.628\dots q^5+\dots.
\end{split}
\end{displaymath}
This example illustrates Corollary~\ref{cuspidalgoodness} in a situation where there is no $M_f$ which is good for $f$.
\end{example}

As mentioned above, Theorem~\ref{mainthm} relies critically on the theory of holomorphic projections, harmonic
Maass forms, and the combinatorial properties of Rankin-Cohen brackets.
In Section~\ref{nutsandbolts} we recall the essential properties of the theory of harmonic Maass forms,
Maass-Poincar\'e series, and Rankin-Cohen brackets.
In Section~\ref{Lfcndef} we define the relevant Dirichlet series and $q$-series
$\L^{(\nu)}(f_1,f_2;\tau)$. To obtain Theorem~\ref{mainthm}
we must modify the existing theory of holomorphic projections to obtain a {\it regularized holomorphic
projection}. This regularization is required for the general case of Theorem~\ref{mainthm}.
We give this in Section~\ref{proofs} by suitably modifying earlier work of Gross and Zagier.
 We then conclude with the proofs of the main
results of this paper.

\section*{Acknowledgements}
The authors thank Jeff Hoffstein and Thomas Hulse for some discussions concerning their work on shifted convolution $L$-functions, and the first author thanks Kathrin Bringmann, Don Zagier and Sander Zwegers for useful conversations related to harmonic Maass forms. The authors 
would also like to thank the anonymous referee for several helpful suggestions which helped to improve the outline of this paper.

Most of this work was carried out while the first author was visiting Emory University. For the financial support he thanks the DFG Graduiertenkolleg 1269 ``Global Structures in Geometry and Analysis'' at the University of Cologne. The second author thanks the generous support of the National Science Foundation and the Asa Griggs Candler Fund.

\section{Nuts and Bolts}\label{nutsandbolts}

\subsection{Harmonic Maass forms}

We recall some basic facts about harmonic Maass forms. These real-analytic modular forms were introduced by Bruinier and Funke in \cite{BF04}. This theory explains
the role of Ramanujan's enigmatic {\it mock theta functions} in the theory of automorphic forms
(see \cite{Ono08, ZagierBourbaki, ZwegersDiss}).

Throughout, we assume that $k\geq 2$ is even, and we let $\tau=x+iy \in\H$, where $x,\,y\in\R$. We define the weight $k$ \emph{slash operator} acting on smooth functions $f:\H\rightarrow\C$ by
\[(f|_k\gamma)(\tau):=(c\tau+d)^{-k}f\left(\tfrac{a\tau+b}{c\tau+d}\right),\]
where $\gamma=\kabcd\in\SL_2(\R)$. The weight $k$ \emph{hyperbolic Laplacian} is defined by
\[\Delta_k:=-y^2\left(\frac{\partial^2}{\partial x^2}+\frac{\partial^2}{\partial y^2}\right)+iky\left(\frac{\partial}{\partial x}+i\frac{\partial}{\partial y}\right).\]

\begin{definition}\label{def:harmMF}
A smooth function $f:\H\rightarrow\C$ is called a \emph{harmonic weak Maass form}\footnote{For convenience, we will usually use the term \emph{harmonic Maass form} and omit the word ``weak''.} of \emph{weight} $2-k$ on
 $\Gamma_0(N)$ if the following conditions hold:
\begin{enumerate}
\item $f$ is invariant under the action of $\Gamma_0(N)$, i.e. for all $\gamma=\kabcd\in\Gamma_0(N)$ we have
\[(f| _{2-k}\gamma)(\tau)=f(\tau)\]
 for all $\tau\in\H$.
\item $f$ lies in the kernel of the weight $2-k$ hyperbolic Laplacian, i.e.
\[\Delta_{2-k}f\equiv 0.\]
\item $f$ grows at most linearly exponentially at the cusps of $\Gamma_0(N)$.
\end{enumerate}
The $\C$-vector space of harmonic weak Maass forms of weight $2-k$ on $\Gamma_0(N)$ is denoted by $H_{2-k}(\Gamma_0(N))$.
\end{definition}
The following property of these functions is well known (for example, see equations $(3.2a)$ and $(3.2b)$ in \cite{BF04}).
\begin{lemma}\label{lem:split}
Let $f\in H_{2-k}(\Gamma_0(N))$ be a harmonic Maass form. Then there is a canonical splitting
\begin{equation}\label{eq:split}
f(\tau)=f^+(\tau)+\frac{(4\pi y)^{1-k}}{k-1}\overline{c_f^-(0)}+f^-(\tau),
\end{equation}
where for some $m_0,n_0\in\Z$ we have the Fourier expansions
\[f^+(\tau):=\sum\limits_{n=m_0}^\infty c_f^+(n)q^n,\]
and
\[f^-(\tau):=\sum\limits_{\substack{n=n_0 \\n\neq 0}}^\infty \overline{c_f^-(n)}n^{k-1}\Gamma(1-k;4\pi ny)q^{-n},\]
where $\Gamma(\alpha;x)$ denotes the usual incomplete Gamma-function.
\end{lemma}
The series $f^+$ in the above lemma is the \emph{holomorphic part} of $f$, and it is known as a \emph{mock modular form} when the \emph{non-holomorphic part} $\tfrac{(4\pi y)^{1-k}}{k-1}\overline{c_f^-(0)}+f^-(\tau)$ doesn't vanish. Products of mock modular forms and usual modular forms, as well as linear combinations of such functions, are called \emph{mixed mock modular forms}.

The explanation for the complex conjugation of the coefficients in the non-holomorphic part $\tfrac{(4\pi y)^{1-k}}{k-1}\overline{c_f^-(0)}+f^-(\tau)$ of $f$ is given in the following proposition due to Bruinier and Funke (see Proposition 3.2 and Theorem 3.7 in \cite{BF04}).
\begin{proposition}
The operator
\[\xi_{2-k}:H_{2-k}(\Gamma_0(N))\rightarrow M_k^!(\Gamma_0(N)),\:f\mapsto\xi_{2-k}f:=2iy^{2-k}\overline{\frac{\partial f}{\partial\overline{\tau}}}\]
is well-defined and surjective with kernel $M_{2-k}^!(\Gamma_0(N))$, the space of weakly holomorphic modular forms of weight $2-k$ on $\Gamma_0(N)$. Moreover, we have that
\[(\xi_{2-k}f)(\tau)=-(4\pi)^{k-1}\sum\limits_{n=n_0}^\infty c_f^-(n)q^n.\]
\end{proposition}
The function $-(4\pi)^{1-k}\xi_{2-k}f$ is called the \emph{shadow} of the mock modular form\footnote{We shall also call this the shadow of the harmonic Maass form $f$} $f^+$. For our purpose, we only consider cases, where the shadow of $f^+$ is a cusp form. Since the $\xi$-operator is surjective, the natural question arises on how to construct a harmonic Maass form with prescribed shadow. This can be done using Poincar\'e series which we recall next.

\subsection{Poincar\'e series}\label{poincareseries}

Here we recall the classical cuspidal Poincar\'e series and the harmonic Maass-Poincar\'e series.
We then give their relationship under the $\xi$-operator described above.

A general Poincar\'e series of weight $k$ for $\Gamma_0(N)$ is given by
\[\P(m,k,N,\varphi_m;\tau):=\sum\limits_{\gamma\in\Gamma_\infty\setminus \Gamma_0(N)}(\varphi_m^*|_k\gamma)(\tau),\]
where $m$ is an integer, $\Gamma_\infty:=\left\{\pm\left(\begin{smallmatrix} 1 & n \\ 0 & 1 \end{smallmatrix}\right)\: :\: n\in\Z \right\}$ is the subgroup of translations in $\Gamma_0(N)$, and $\varphi_m^*(\tau):=\varphi_m(y)\exp(2\pi imx)$ for a function $\varphi_m:\R_{>0}\rightarrow\C$ which is $O(y^A)$ as $y\rightarrow 0$ for some $A\in\R$. We
distinguish two special cases ($m>0$),
\begin{align}
P(m,k,N;\tau)&:=\P(m,k,N,e^{-my};\tau)\\
Q(-m,k,N;\tau)&:=\P(-m,2-k,N,\calM_{1-\frac k2}(-4\pi my);\tau),
\end{align}
where $\calM_s(y)$ is defined in terms of the $M$-Whittaker function. We often refer to $Q(-m,k,N;\tau)$ as a \emph{Maass-Poincar\'e series}.

Next we give the Fourier expansions of the Poincar\'e series $P(m,k,N;\tau)$ and $Q(-m,k,N;\tau)$.
In order to do so, let us define the \emph{Kloosterman sums} by
\begin{equation}
K(m,n,c):=
\sum_{v(c)^{\times}} e\left(\frac{m\overline
v+nv}{c}\right),
\end{equation}
where $e(\alpha):=e^{2\pi i \alpha}$. The
 sum over $v$ runs through the primitive residue classes
modulo $c$, and $\overline v$ denotes the multiplicative inverse
of $v$ modulo $c$.

The Fourier expansions of the cuspidal Poincar\'e series (for
example, see \cite{Iwaniecbook})
are described in terms of infinite sums of Kloosterman sums weighted by
$J$-Bessel functions.

\begin{lemma}\label{Hfourier} If $k\geq 2$ is even and $m, N\in \N$, then the Poincar\'e series $P(m,k,N;\tau)$ is in 
$S_k(\Gamma_0(N))$ with a Fourier expansion
\begin{equation}
P(m,k,N;\tau)=:q^{m}+\sum_{n=1}^{\infty} a(m,k,N;n)q^n,
\end{equation}
with
\begin{displaymath}
a(m,k,N;n)=2\pi
(-1)^{\frac{k}{2}}\left(\frac{n}{m}\right)^{\frac{k-1}{2}} \cdot
\sum_{\substack{c>0\\c\equiv 0\pmod{N}}} \frac{K(m,n,c)}{c}\cdot
J_{k-1} \left(\frac{4\pi \sqrt{mn}}{c}\right).
\end{displaymath}
\end{lemma}

Now we recall one family of Maass-Poincar\'e series which have
appeared in earlier works (for example, see \cite{BringOPNAS, Bruinier, Fay,
Hejhal, Niebur1}).

\begin{lemma}\label{Ffourier} If $k\geq 2$ is even and $m, N\geq 1$, then
$Q(-m,k,N;\tau)$ is in $H_{2-k}(\Gamma_0(N))$. Moreover, if $I_{k-1}$ is the
usual $I$-Bessel function, then
$$
Q(-m,k,N;\tau) = (1-k) \left( \Gamma(k-1,4\pi my) - \Gamma(k-1) \right)
\, q^{-m} + \sum_{n\in\Z} c_m(n,y) \, q^n.
$$

\noindent 1) If $n<0$, then
\begin{displaymath}
\begin{split}
c_m(n,y)
  =2 \pi i^{k}  (1-k) \, &\Gamma(k-1,4 \pi |n| y)
  \left|\frac{n}{m}\right|^{\frac{1-k}{2}}\\
 &\ \ \ \ \times  \sum_{\substack{c>0\\c\equiv 0\pmod{N}}}
   \frac{K(-m,n,c)}{c}\cdot
 J_{k-1}\!\left(\frac{4\pi\sqrt{|mn|}}{c}\right).
\end{split}
\end{displaymath}

\noindent
2) If $n>0$, then
$$
c_m(n,y)= - 2 \pi  i^k \Gamma(k)   \left( \frac{n}{m}
\right)^{\frac{1-k}{2}}
  \sum_{\substack{c>0\\c\equiv 0\pmod{N}}}
   \frac{K(-m,n,c)}{c}\cdot
 I_{k-1}\!\left(\frac{4\pi\sqrt{|mn|}}{c}\right).
$$

\noindent 3) If $n=0$, then
$$
c_m(0,y)=-(2\pi i)^{k} m^{k-1} \sum_{\substack{c>0\\c\equiv
0\pmod{N}}}
 \frac{K(-m,0,c)}{c^k}.
$$
\end{lemma}

These two families of Poincar\'e series are closely related (for example, see  Theorem 1.1 in \cite{BringOPNAS}).

\begin{lemma}\label{poincarerelationships}
If $k\geq 2$ is even and $m, N\geq 1$, then
\begin{displaymath}
\xi_{2-k}(Q(-m,k,N;\tau))= ( 4\pi)^{k-1} m^{k-1}(k-1)\cdot
P(m,k,N;\tau)\in S_k(\Gamma_0(N)).
\end{displaymath}
\end{lemma}
\begin{proof} The claim follows easily from
the explicit  expansions in Lemma~\ref{Hfourier} and
\ref{Ffourier} and the definition of $\xi_{2-k}$.
\end{proof}

We conclude this section with a well-known result about the behaviour of the Maass-Poincar\'e series at the cusps (see Proposition 3.1 in \cite{BJO06}).
\begin{lemma}\label{growthQ}
For $k\geq 2$ and $m,N\geq 1$, the harmonic Maass form $Q(-m,k,N;\tau)$ grows like $\exp(2\pi m y)$ as $\tau$ approaches the cusp $i\infty$ and grows moderately approaching any other cusp of $\Gamma_0(N)$. The cuspidal Poincar\'e series $P(m,k,N;\tau)$ decays like $\exp(-2\pi y)$ as $\tau$ approaches the cusp $i\infty$, and (linearly) exponentially approaching any other cusp. 
\end{lemma}

\subsection{Rankin-Cohen Brackets}
Rankin-Cohen brackets are a generalization of the usual product of modular forms. They are bilinear differential operators which map modular forms to modular forms. They were introduced independently by Rankin in \cite{Rankin2} and Cohen in \cite{Coh75}. A very nice discussion of them may be found in \cite{ZagierRCB}.
\begin{definition}\label{def:RCB}
Let $f,g:\H\rightarrow\C$ be smooth functions on the upper half-plane and $k,\ell\in\R$ be some real numbers, the \emph{weights} of $f$ and $g$. Then for a non-negative integer $\nu$ we define the $\nu$th \emph{Rankin-Cohen bracket} of $f$ and $g$ by
\[[f,g]_\nu:=\frac{1}{(2\pi i)^\nu}\sum\limits_{\mu=0}^\nu (-1)^\mu {{k+\nu-1} \choose \nu-\mu} {{\ell+\nu-1} \choose \mu} \frac{\partial^\mu f}{\partial\tau^\mu}\frac{\partial^{\nu-\mu}g}{\partial\tau }.\]
\end{definition}
The following result (see Theorem 7.1 in \cite{Coh75}) gives the connection of this definition to modular forms.
\begin{theorem}[Cohen]
Let $f,g$ as in Definition~\ref{def:RCB} with integral weights $k,\ell$. Then the $\nu$th Rankin-Cohen bracket commutes with the slash operator, i.e.
\[[(f|_k\gamma),(g|_\ell\gamma)]_\nu=([f,g]_\nu)|_{k+\ell+2\nu}\gamma\]
for all $\gamma\in\SL_2(\R)$. In particular, if $f$ and $g$ are modular of their respective weights, then $[f,g]_\nu$ is modular of weight $2\nu+k+\ell$.
\end{theorem}
\begin{remark}
The $0$th Rankin-Cohen bracket of $f$ and $g$ is just their usual product, the first bracket defines a Lie bracket on the graded algebra of real-analytic modular forms, giving it the structure of a so-called \emph{Poisson algebra}.
\end{remark}
\begin{remark}
The $\nu$th Rankin-Cohen bracket is symmetric in $f$ and $g$ if $\nu$ is even, otherwise it is antisymmetric.
\end{remark}
Since the Rankin-Cohen bracket provides a ``product'' on the algebra of modular forms, it is natural to consider functions of the form $[f^+,g]_\nu$ for $f\in H_{2-k}(\Gamma_0(N))$ and $g\in M_\ell(\Gamma_0(N))$ to be \emph{mixed mock modular forms} as well.

An interesting observation, that doesn't seem to appear in the literature, is the following property of Rankin-Cohen brackets of harmonic Maass forms and holomorphic modular forms.
\begin{lemma}\label{Maass}
Let $f\in H_{2-k}(\Gamma_0(N))$ and $g\in M_\ell(\Gamma_0(N))$ for even, positive integers $k,\ell$. Then for $\nu>k-2$ the Rankin-Cohen bracket $[f,g]_\nu$ is a weakly holomorphic modular form of weight $2\nu-k+\ell+2$.
\end{lemma}
\begin{proof}
Consider the operators
\begin{align}
\label{eq:raising}
R_k&:=2i\frac{\partial}{\partial\tau}+ky^{-1},\\
\label{eq:lowering}
L_k&:=-2iy^2\frac{\partial}{\partial\overline{\tau}},
\end{align}
known as the \emph{Maass raising operator}, resp. \emph{Maass lowering operator}. These operators map modular forms of weight $k$ to modular forms of weight $k+2$ (resp. $k-2$).

An explicit description of the iterated raising operator $R_k^n:=R_{k+2(n-1)}\circ\dots R_{k+2}\circ R_k$ for $n>0$, $R_k^0:=\operatorname{id}$ is given by (see \cite{LZ01}, Equation (4.15))
\[R_k^n=\sum\limits_{m=0}^n {n \choose m} \frac{\Gamma(k+n)}{\Gamma(k+n-m)}y^{-m}(2i)^{n-m}\frac{\partial^{n-m}}{\partial\tau^{n-m}},\]
which follows easily by induction on $n$.

With this one obtains by straightforward calculation that if $f\in H_{2-k}(\Gamma_0(N))$ and $g\in M_\ell(\Gamma_0(N))$, then we have that
\[L_{2\nu-k+\ell+2}([f,g]_\nu)=\frac{1}{(4\pi)^\nu}{{k-2} \choose {\nu}} L_{2-k}(f)R_\ell^\nu (g).\]
But for $k\in\N$, this is $0$ whenever $\nu>k-2$, thus $[f,g]_\nu$ must be holomorphic.
\end{proof}

\section{Holomorphic Projection and the Proof of Theorem~\ref{mainthm}}\label{proofs}

Here we prove Theorem~\ref{mainthm}. We begin by recalling the principle of
{\it holomorphic projection.}

\subsection{Holomorphic Projection}\label{sec:holproj}

The basic idea behind holomorphic projection is the following. Suppose you have a weight $k$ real-analytic modular form $\tilde{f}$ on $\Gamma_0(N)$ with moderate growth at the cusps. Then this function defines a linear functional on the space of weight $k$ holomorphic cusp forms by $g\mapsto \langle g,\tilde{f}\rangle$, where $\langle\cdot,\cdot\rangle$ denotes the Petersson scalar product. But this functional must be given by $\langle\cdot,f\rangle$ for some holomorphic cusp form $f$ of the same weight. This $f$ is essentially the holomorphic projection of $f$. In \cite{St80}, Sturm introduced this notion, which was used and further developed later for example in the seminal work of Gross and Zagier on Heegner points and derivatives of $L$-function \cite{GrZ86}. The technique is also used in order to obtain recurrence relations among Fourier coefficients of mock modular forms, see e.g. \cite{IRR13, Mertpreprint}.

Here we briefly recall this framework for holomorphic projections.

\begin{definition}\label{def:pihol}
Let $f:\H\rightarrow\C$ be a (not necessarily holomorphic) modular form of weight $k$ on $\Gamma_0(N)$ with a Fourier expansion
\[f(\tau)=\sum\limits_{n\in\Z} a_f(n,y)q^n,\]
where again $y=\Im(\tau)$. For a cusp $\kappa_j$, $j=1,...,M$ and $\kappa_1:=i\infty$, of $\Gamma_0(N)$ fix $\gamma_j\in\SL_2(\Z)$ with $\gamma_j\kappa_j=i\infty$. Suppose that for some $\eps>0$ we have
\begin{enumerate}
\item $f(\gamma_j^{-1}w)\left(\frac d{dw}\tau\right)^\frac k2=c_0^{(j)}+O(\Im(w)^{-\eps})$ as $w\rightarrow i\infty$
for all $j=1,...,M$ and $w=\gamma_j\tau$,
\item $a_f(n,y)=O(y^{2-k})$ as $y\rightarrow 0$ for all $n>0$.
\end{enumerate}
Then we define the \emph{holomorphic projection} of $f$ by
\[
(\pihol f)(\tau):=(\pihol^kf)(\tau):=c_0+\sum\limits_{n=1}^\infty c(n)q^n,
\]
with $c_0=c_0^{(1)}$ and
\begin{equation}\label{eq:holproj}
\begin{aligned}
c(n)=\frac{(4\pi n)^{k-1}}{(k-2)!}\int_0^\infty a_f(n,y)e^{-4\pi ny}y^{k-2}dy
\end{aligned}
\end{equation}
for $n>0$.
\end{definition}
This operator has several nice properties, some of which are summarized in the following proposition.
\begin{proposition}\label{prop:pihol}
Let $f$ be as in Definition~\ref{def:pihol}. Then the following are true.
\begin{enumerate}
\item If $f$ is holomorphic, then $\pihol f=f$.
\item The function $\pihol f$ lies in the space $\widetilde{M}_k(\Gamma_0(N))$.
\end{enumerate}
\end{proposition}
\begin{proof}
Claim $(1)$ of this proposition is an easy and elementary calculation, see e.g. Proposition 3.2 of \cite{IRR13}, while the first proof of $(2)$ using Poincar\'e series (and ``Hecke's trick'' for the case of weight $2$) was given in Proposition 5.1 and Proposition 6.2 in \cite{GrZ86}. A different proof which uses spectral methods and the language of vector-valued modular forms is given in Theorem 3.3 in \cite{IRR13}.
\end{proof}

Proposition~\ref{prop:pihol}, when combined with the previous results on Rankin-Cohen brackets and harmonic Maass forms
is sufficient for proving Theorem~\ref{mainthm} when $M_{f_1}$ is good for $f_2$. As mentioned in the introduction,
this situation is quite rare. Therefore, we must extend this notion to accomodate the most general cases.
To this end we introduce {\it regularized holomorphic projections}.

\subsection{Regularized holomorphic projections}

The classical holomorphic projection is constructed to respect the Petersson inner product on the space of cusp forms. We construct a regularized holomorphic projection by the same recipe. This projection shall respect the regularized Petersson inner product introduced by Borcherds in \cite{Borcherds}.
\begin{definition}\label{defpiholreg}
Let $f:\H\rightarrow\C$ be a real-analytic modular form of weight $k\geq 2$ with a Fourier expansion
\[f(\tau)=\sum\limits_{n\in\Z} a_f(n,y)q^n.\]
Suppose further that for each cusp $\kappa$ of $\Gamma_0(N)$ there exists a polynomial $H_\kappa(X)\in\C[X]$, such that we have
\[(f|_k\gamma_\kappa^{-1})(\tau)-H_\kappa(q^{-1})=O(y^{-\eps})\]
for some $\eps>0$ and further suppose that $a_f(n,y)=O(y^{2-k})$ as $y\rightarrow 0$ for all $n>0$. Then we define the {\it regularized holomorphic projection} of $f$ by
$$
(\pihol^{reg}f)= H_{i\infty}(q^{-1}) +\sum\limits_{n=1}^\infty c(n)q^n,
$$
where
\begin{equation}\label{eqpiholreg}
c(n)=\lim\limits_{s\rightarrow 0}\frac{(4\pi n)^{k-1}}{(k-2)!}\int_0^\infty a_f(n,y)e^{-4\pi ny}y^{k-2-s}dy.
\end{equation}
\end{definition}
This operator enjoys a natural analog of Proposition~\ref{prop:pihol}.

\begin{proposition}\label{prop:piholreg}
Let $f$ be a function as in Definition~\ref{defpiholreg}.
\begin{enumerate}
\item If $f$ is holomorphic on $\H$, then we have $\pihol^{reg}f=f$.
\item The function $\pihol^{reg}f$ lies in the space $\widetilde{M}^!_k(\Gamma_0(N))$.
\item If $f$ satisfies Definition~\ref{def:pihol}, then
$\pihol^{reg}(f)=\pihol(f)$.
\end{enumerate}
\end{proposition}
\begin{proof}
Claims (1) and (3) are clear.
The proof of (2)
is a modification of the proofs of Propositions 5.1 and 6.2 in \cite{GrZ86} which concern the classical holomorphic projection.
 Here we indicate how to modify these proofs to this setting.

We begin by recalling the {\it regularized Petersson inner product}
defined by Borcherds \cite{Borcherds}. For $T>0$, we denote by $\calF_T(\SL_2(\Z))$ the truncated version of the usual fundamental domain of $\SL_2(\Z)$, i.e.
\[\calF_T(\SL_2(\Z)):=\left\{\tau\in\H\ :\ |\tau|\geq 1,\ |\Re(\tau)|\leq \frac 12,\ \Im(\tau)\leq T\right\}.\]
For a finite index subgroup $\Gamma\leq\SL_2(\Z)$ we define a truncated fundamental domain of $\Gamma$ as
\[\calF_T(\Gamma):=\bigcup\limits_{\gamma\in\calV}\gamma\calF_T(\SL_2(\Z)),\]
where $\calV$ is a fixed set of representatives of $\Gamma\setminus\SL_2(\Z)$. The regularized inner product of $g\in M_k(\Gamma)$ and $h$ a weight $k$ real-analytic modular form on $\Gamma$ ($k\geq 2$ even) with at most linearly exponential growth at the cusps is given as the constant term in the Laurent expansion of the meromorphic continuation of the expression
\[\langle g,h\rangle^{reg}:=\frac{1}{[\SL_2(\Z):\Gamma]}\lim\limits_{T\rightarrow\infty}\int\limits_{\calF_T(\Gamma)}g(\tau)\overline{h(\tau)
}y^{k-2s}\frac{dxdy}{y^2}.\]
It follows from Corollary 4.2 in \cite{BOR} that the weakly holomorphic Poincar\'e series from Section~\ref{poincareseries} together with the Eisenstein series (including the quasimodular Eisenstein series $E_2$ in case that $k=2$) are orthogonal to cusp forms with respect to $\langle\cdot,\cdot\rangle^{reg}$. By subtracting a suitable linear combination of these functions one can replace $f$ by a function that satisfies the growth condition in Proposition 5.1 in \cite{GrZ86}. The proof then follows \emph{mutatis mutandis}
as in \cite{GrZ86}.

We consider this product for $\Gamma=\Gamma_0(N)$, where $g$ is the regularized Poincar\'e series
\[P_{m,s}(\tau):=\sum\limits_{\gamma\in\Gamma_\infty\setminus\Gamma_0(N)}(y^se^{2\pi im\tau})|_k\gamma,\]
where $m\geq 1$ and $\Gamma_\infty:=\left\{\pm\left(\begin{smallmatrix} 1 & n \\ 0 & 1\end{smallmatrix}\right)\ ,\ n\in\Z\right\}$ denotes the stabilizer of the cusp $i\infty$ in $\Gamma_0(N)$.

Note that this regularization is necessary to ensure convergence in the case of weight $2$, if the weight is at least $4$, these converge to the cuspidal Poincar\'e series in Section~\ref{poincareseries} as $s\rightarrow 0$. Since the domain of integration, the truncated fundamental domain, is compact and all functions in the integrand are continuous, it is legitimate to replace $P_{m,s}$ by its definition as a series and interchange summation and integration. This yields ($\gamma=\smallabcd$)
\begin{align*}
\langle g,h\rangle^{reg}&=\frac{1}{[\SL_2(\Z):\Gamma]}\cdot \lim\limits_{T\rightarrow\infty}\sum\limits_{\gamma\in\Gamma_\infty\setminus\Gamma_0(N)}
\int\limits_{\calF_T(\Gamma_0(N))} e^{2\pi im\frac{a\tau+b}{c\tau+d}}\overline{h\left(\frac{a\tau+b}{c\tau+d}\right)}\frac{y^{k-s}}{|c\tau+d|^{2(k-s)}}\frac{dxdy}{y^2}\\
&=\frac{1}{[\SL_2(\Z):\Gamma]}\cdot \lim\limits_{T\rightarrow\infty}\int\limits_{\calF_T(\Gamma_\infty)} e^{2\pi im\tau}\overline{h\left(\tau\right)}y^{k-s}\frac{dxdy}{y^2}\\
&=\frac{1}{[\SL_2(\Z):\Gamma]}\cdot \int\limits_{\Gamma_\infty\setminus\H} e^{2\pi im\tau}\overline{h\left(\tau\right)}y^{k-s}\frac{dxdy}{y^2}.
\end{align*}
From here, the proof literally is the same as in \cite{GrZ86}, so we refer the reader to there.
\end{proof}

\subsection{Holomorphic projections of completed mixed mock modular forms}

For the proof of our main result, we need to know the action of the holomorphic projection operator on Rankin-Cohen brackets of harmonic Maass forms and cusp forms. The following result (see Theorem 3.6 in \cite{Mertpreprint}) gives an explicit formula for the Fourier coefficients of such a holomorphic projection. The case of ordinary products, i.e. $\nu=0$, is already contained in Theorem 3.5 in \cite{IRR13}. Before the statement of the proposition, let us define the polynomial
\begin{equation}\label{eq:P}
G_{a,b}(X,Y):=\sum\limits_{j=0}^{a-2} (-1)^j{{a+b-3} \choose {a-2-j}}{{j+b-2} \choose j} X^{a-2-j}Y^j\in\C[X,Y]
\end{equation}
for integers $a\geq 2$ and $b\geq 0$, which is homogeneous of degree $a-2$.
\begin{proposition}\label{theo:piholRCB}
Let $f_1\in S_{k_1}(\Gamma_0(N))$ and $f_2\in S_{k_2}(\Gamma_0(N))$ be cusp forms of even weights $k_1\geq k_2$ as in the introduction and let $M_{f_1}\in H_{2-k_1}(\Gamma_0(N))$ be a harmonic Maass form with shadow $f_1$.
%which is good for $f_2$. 
%If $[M_{f_1},f_2]_\nu$ satisfies the conditions in Definition~\ref{defpiholreg}, t
%N.B. The conditions in the definition are always satisfied in this case.
Then we have for $0\leq \nu\leq \tfrac{k_1-k_2}2$ that
\begin{equation}\label{eq:piholRCB}
\begin{aligned}
&\pihol^{reg}([M_{f_1},f_2]_\nu)(\tau)=[M_{f_1}^+,f_2]_\nu(\tau)-(k_1-2)!\sum\limits_{h=1}^\infty q^h\left[\sum\limits_{n=1}^\infty a_2(n+h)\overline{a_1(n)}\right.\\
& \left.\times\sum\limits_{\mu=0}^\nu \left(\begin{smallmatrix}{\nu-k_1+1} \\ {\nu-\mu} \end{smallmatrix}\right)\left(\begin{smallmatrix}{\nu+k_2-1} \\ \mu\end{smallmatrix}\right)\Big((n+h)^{-\nu-k_2+1}G_{2\nu-k_1+k_2+2,k_1-\mu}(n+h,n)-n^{\mu-k_1+1}(n+h)^{\nu-\mu} \Big) \right].
     \end{aligned}
     \end{equation}
\end{proposition}

\begin{proof}
We write
\[M_{f_1}^-(\tau)=\sum\limits_{n=1}^\infty n^{1-k_1}\overline{a_1(n)}\Gamma^*(k_1-1,4\pi ny)\overline{q}^{-n}\]
with $\Gamma^*(\alpha;x):=e^x\Gamma(\alpha;x)$. This representation makes it easy to determine the $\mu$th derivative of $M_{f_1}^-$ as
\[\frac{1}{(2\pi i)^\mu}\left(\frac{\partial^\mu}{\partial\tau^\mu}M_{f_1}^-\right)(\tau)=(-1)^{\mu}\frac{\Gamma(k_1-1)}{\Gamma(k_1-\mu-1)}\sum\limits_{n=1}^\infty n^{\mu-k_1+1}\overline{a_1(n)}\Gamma(k_1-\mu-1;4\pi ny)q^{-n}.\]

Using this, we find that
\[[M_{f_1}^-,f_2]_\nu(\tau)=\sum\limits_{h\in\Z} b(h,y)q^h,\]
where
\begin{align*}
b(h,y)=&\sum\limits_{n=1}^\infty\sum\limits_{\mu=0}^\nu {{1-k_1+\nu} \choose \nu-\mu} {{k_2+\nu-1} \choose \mu}\\
       &\qquad\qquad\qquad\times\frac{\Gamma(k_1-1)}{\Gamma(k_1-\mu-1)}n^{1-k_1+\mu}\overline{a_1(n)}\Gamma(k_1-\mu-1;4\pi ny)a_2(n+h)(n+h)^{\nu-\mu}.
\end{align*}
Thus the holomorphic projection of this becomes
\[\pihol^{reg}([f^-,g]_\nu)=\sum\limits_{h=1}^\infty b(h)q^h,\]
with
\begin{align*}
b(h)=&\frac{(4\pi h)^{2\nu-k_1+k_2+1}}{(2\nu-k_1+k_2)!}\sum\limits_{n=1}^\infty\sum\limits_{\mu=0}^\nu {{1-k_1+\nu} \choose \nu-\mu} {{k_2+\nu-1} \choose \mu}\\
     &\qquad\qquad\qquad\times\frac{\Gamma(k_1-1)}{\Gamma(k_1-\mu-1)}n^{1-k_1+\mu}\overline{a_1(n)}a_2(n+h)(n+h)^{\nu-\mu}\\
     &\qquad\qquad\qquad\times\int\limits_0^\infty \Gamma(k_1-\mu-1;4\pi ny)e^{-4\pi hy}y^{2\nu-k_1+k_2}dy.
\end{align*}
For the evaluation of the integral we may interchange the outer integration and the implicit one in the definition of the incomplete Gamma function, which after several substitutions of variables and simplification steps, which are carried out in detail in the proof of Lemma 3.7 in \cite{Mertpreprint}, yields the claim. We point out that none of these steps actually changes the order of summation, so also the case of conditional convergence ($\nu=\tfrac{k_1-k_2}2$) works fine.
\end{proof}
%\begin{remark}
%A necessary condition for convergence is obviously that $\nu\leq \tfrac{k_1-k_2}{2}$.
%\end{remark}
\begin{remark}
In the case $k_1=k_2=k$ and $\nu=0$, equation \eqref{eq:piholRCB} simplifies to
\[\pihol^{reg}(M_{f_1}\cdot f_2)(\tau)=M_{f_1}^+(\tau)\cdot f_2(\tau)-(k-2)!\sum\limits_{h=1}^\infty\left[\sum\limits_{n=1}^\infty a_2(n+h)\overline{a_1(n)}\left(\frac{1}{(n+h)^{k-1}}-\frac{1}{n^{k-1}}\right)\right]q^h.\]
\end{remark}

\subsection{The Dirichlet series}\label{Lfcndef}

Here we define the general Dirichlet series in Theorem~\ref{mainthm}.
We recall from the introduction the
{\it derived} shifted convolution series
$$
D^{(\mu)}(f_1, f_2,h;s):=
\sum_{n=1}^{\infty}\frac{a_{1}(n+h)\overline{a_{2}(n)}(n+h)^{\mu}}{n^s}.
$$
Obviously, we have that $D(f_1,f_2,h;s)=D^{(0)}(f_1,f_2,h;s)$.
Due to the absence of the symmetry in the $f_i$, it is natural to consider special values
of the symmetrized Dirichlet series
\begin{equation}\label{Dhat}
\Dhat(f_1, f_2,h;s):=D(f_1,f_2,h;s)-\delta_{k_1,k_2}D(\overline{f_2},\overline{f_1},-h;s),
\end{equation}
where we set $a_j(n):=0$ for $n\leq 0$ and $\delta_{ij}$ denotes the usual Kronecker $\delta$.
%We note that when $k_2<k_1$ the sum is empty. 
As we shall see, one important consequence of symmetrizing these functions is the convergence of the special values
that we consider here. %For the case of equal weights, the above definition simplifies to
%\[\Dhat(f_1,f_2,h;s)=D(f_1,f_2,h;s)-D(\overline{f_2},\overline{f_1},-h;s).\]
Note that $\Dhat(f_1,f_2,h;s)$ is, as a Dirichlet series, in fact conditionally convergent at $s=k_1-1$. This can be 
seen immediately from the estimate 
\[\sum\limits_{m=1}^Ma_1(m+h)\overline{a_2(m)}\ll_\eps M^{\frac{k_1+k_2}{2}-\delta}\]
for $h\leq M^{\frac{4}{3}-\eps}$ ($\eps>0$) and some $\delta>0$, see Corollary 1.4 in \cite{Blomer}.

\noindent\emph{Remark:} We note that Blomer chose a different normalization and only considered shifted convolution sums for a single cusp form (i.e. $a_1(n)=a_2(n)$ for all $n$). However, an inspection of the proof of Theorem 1.3 in \cite{Blomer} shows that all his arguments carry over directly to our case. 

For $\nu=\tfrac{k_1-k_2}{2}$, we consider the generating function in $h$-aspect
of the symmetrized shifted convolution special values of the Dirichlet series, convergent for $\Re(s)>k_1$,

\begin{equation}\label{Dhatnu}
\begin{aligned}
\Dhat^{(\nu)}(f_1,f_2,h;s):=\sum_{\mu=0}^{\nu}&\alpha(\nu,k_1, k_2,\mu)D^{(\nu-\mu)}(f_1,f_2,h;s-\mu)\\
  &\ \qquad-\beta(\nu,k_1,k_2)D^{(0)}(\overline{f_2},\overline{f_1},-h;s-\nu).
\end{aligned}
\end{equation}
The $\alpha(\nu,k_1, k_2,\mu)$ and $\beta(\nu,k_1,k_2)$ are integers defined by
\begin{equation}\label{alpha}
\alpha(\nu,k_1,k_2,\mu):={{\nu-k_1+1} \choose {\nu-\mu}}{{\nu+k_2-1} \choose \mu}
\end{equation}
and
\begin{equation}\label{beta}
\beta(\nu,k_1,k_2):=\sum\limits_{\mu=0}^\nu {{\nu-k_1+1} \choose {\nu-\mu}}{{\nu+k_2-1} \choose \mu}.
\end{equation}
Note that again, viewed as a Dirichlet series, $\Dhat^{(\nu)}$ is conditionally convergent at $s=k_1-1$.

Of course we have that $\Dhat(f_1,f_2,h;s)=\Dhat^{(0)}(f_1,f_2,h;s)$.
The generating function we study is
\begin{equation}\label{genfcn}
\L^{(\nu)}(f_1,f_2;\tau):=\sum_{h=1}^{\infty}
 \Dhat^{(\nu)}(f_1,f_2,h;k_1-1)q^h.
\end{equation}
In the special case when $\nu=0$, we have that
$\L^{(0)}(f_1,f_2;\tau)=\L(f_1,f_2;\tau)$.

\subsection{Proof of Theorem~\ref{mainthm} and Corollary~\ref{cuspidalgoodness}}
We can now prove Theorem~\ref{mainthm}.
 
\begin{proof}[Proof of Theorem~\ref{mainthm}]
If one plugs in the definition of $G_{a,b}$ in \eqref{eq:P} into \eqref{eq:piholRCB} %and rearranges {\bf MM: corrected typo} the sums appropriately, 
one gets that
\begin{align*}
&\pihol^{reg}([M_{f_1},f_2]_\nu)(\tau)=[M_{f_1}^+,f_2]_\nu(\tau)\\
&\quad-(k_1-2)!\sum\limits_{h=1}^\infty q^h \left\lbrace\sum\limits_{n=1}^\infty a_2(n+h)\overline{a_1(n)}\left[\sum\limits_{\mu=0}^\nu\left(\begin{matrix}{\nu-k_1+1} \\ {\nu-\mu}\end{matrix}\right)\left(\begin{matrix}{\nu+k_2-1} \\ \mu\end{matrix}\right)\frac{1}{(n+h)^{k_1-\nu-1}}\right.\right.\\
&\quad\quad\qquad\qquad\qquad \left.\left.-\left(\begin{matrix}{\nu-k_1+1} \\ {\nu-\mu}\end{matrix}\right)\left(\begin{matrix}{\nu+k_2-1} \\ \mu\end{matrix}\right) \frac{(n+h)^{\nu-\mu}}{n^{k_1-\mu-1}}\right]\right\rbrace.
\end{align*}
By definition, this is
\[[M_{f_1}^+,f_2]_\nu(\tau)+(k_1-2)!\cdot\L^{(\nu)}(f_2,f_1;\tau)\]
where $\alpha(\nu,k_1,k_2,\mu)$ and $\beta(\nu,k_1,k_2)$ are
defined by (\ref{alpha}) and (\ref{beta}).

By Proposition~\ref{prop:piholreg}, this function lies in the space $\widetilde{M}^{!}_{2}(\Gamma_0(N))$ which yields Theorem~\ref{mainthm}
in the general case. When $M_{f_1}$ is good for $f_2$, then Proposition~\ref{prop:pihol} implies that it lives
in the space $\widetilde{M}_{2}(\Gamma_0(N))$.
\end{proof}

\begin{proof}[Proof of Corollary~\ref{cuspidalgoodness}]
The formulas stated in the corollary are exactly the same as in Theorem~\ref{mainthm}, keeping in mind that by Lemma~\ref{poincarerelationships} the shadow of $Q$ is precisely $(1-k)\cdot P$, we only have to make sure that $Q$ is good for $P$. But this follows immediately from Lemma~\ref{growthQ} since $P$ decays exponentially in every cusp of $\Gamma_0(N)$ and $Q^+$ has exactly one simple pole at infinity and grows moderately in the other cusps.
\end{proof}


\begin{thebibliography}{BrStr}


\bibitem{Blomer}
V.~Blomer,
\emph{Shifted convolution sums and subconvexity bounds for automorphic $L$-functions},
Int. Math. Res. Notes, \textbf{73} (2004), 3905--3926.

\bibitem{Borcherds}
R.~E.~Borcherds, \emph{Automorphic forms with singularities on Grassmannians},
Invent. Math. \textbf{132} (1998), 491--562.

\bibitem{BringOPNAS} {K. Bringmann and K. Ono},
\emph{Lifting cusp forms to Maass forms with an application to
partitions}, Proc. Natl. Acad. Sci., USA \textbf{104}, No. 10
(2007), pages 3725-3731.


\bibitem{Bruinier} J.\ H.\ Bruinier, \emph{Borcherds products on
  $\Orth(2,l)$ and Chern classes of Heegner divisors},
Springer Lect. Notes Math. {\bf 1780}, Springer-Verlag (2002).


\bibitem{BF04}
J.~H. Bruinier and J.~Funke, \emph{{O}n two geometric theta lifts}, Duke Math.
  J. \textbf{1} (2004), no.~125, 45--90.

\bibitem{BJO06}
J.~H. Bruinier, P. Jenkins, and K. Ono, \emph{Hilbert class polynomials and traces of singular moduli}, Math. Ann. \textbf{334} (2006), 373--393.

\bibitem{BOR}
J.~H. Bruinier, R. Rhoades, and K. Ono, \emph{Differential operators for harmonic weak Maass forms and the vanishing
of Hecke eigenvalues}, Math. Ann. \textbf{342} (2008), 673--693.

\bibitem{Coh75}
H.~Cohen, \emph{Sums {I}nvolving the {V}alues at {N}egative {I}ntegers of
  {$L$}-{F}unctions of {Q}uadratic {C}haracters}, Math. Ann. \textbf{217}
  (1975), 271--285.
  
\bibitem{Fay} J. D. Fay, \emph{Fourier coefficients of the resolvent for a Fuchsian group},
J. Reine Angew. Math. \textbf{293/294} (1977), 143--203.


\bibitem{Good1}  A. Good, \emph{Beitr\"age zur Theorie der Dirichletreihen, die Spitzenformen zugeordnet sind},
J. Number Th. \textbf{13}, (1981), 18--65.

\bibitem{Good2} \bysame, \emph{Cusp forms and eigenfunctions of the Laplacian},
Math. Ann. \textbf{255} (1981), 523--548.

\bibitem{GrZ86}
B.~H. Gross and D.~B. Zagier, \emph{Heegner points and derivatives of
  {$L$}-series}, Invent. Math. \textbf{84} (1986), 225--320.


\bibitem{Hejhal} D. A. Hejhal, \emph{The Selberg trace formula for
$PSL(2,\R)$}, Springer Lect. Notes in Math. \textbf{1001},
Springer-Verlag, Berlin, 1983.

\bibitem{HoffsteinHulse} J. Hoffstein and T. A. Hulse,
\emph{Multiple Dirichlet series and shifted convolutions}, arXiv:1110.4868v2.

\bibitem{IRR13}
\"O. Imamo\u{g}lu, M.~Raum, and O.~Richter.
\emph{Holomorphic projections and {R}amanujan's mock theta functions}, Proc. Nat. Acad. Sci. U.S.A. 111.11 (2014), 3961--3967.


\bibitem{Iwaniecbook} H. Iwaniec, \emph{Topics in classical
automorphic forms}, Grad. Studies in Math. \textbf{17},
Amer. Math. Soc., Providence, RI., 1997.


\bibitem{KZ94} M. Kaneko and D. B. Zagier,
\emph{A generalized Jacobi Theta function and quasimodular forms},
in \emph{The moduli spaces of curves} (R. Dijkgraaf, C. Faber, G. v.d. Geer, eds.), Prog. in Math. \textbf{129}, Birkh\"auser, Boston (1995), 165--172.

\bibitem{Survey} Y.-K. Lau, J. Liu, and Y. Ye, \emph{Shifted convolution sums of Fourier coefficients of cusp forms},
Number Theory, Ser. Number Theory Appl. vol. 2, World Sci. Publ. Hackensack, NJ, 2007,
108--135.

\bibitem{LZ01}
J.~Lewis and D.~Zagier, \emph{Period functions for {M}aass wave forms}, Ann. of
  Math. (2) \textbf{153} (2001), 191--258.

\bibitem{Mertpreprint}
M.~H. Mertens.
\emph{Eichler-{S}elberg {T}ype {I}dentities for {M}ixed {M}ock {M}odular
  {F}orms}, preprint, arXiv:1404.5491.


\bibitem{Niebur1} D. Niebur, \emph{A class of nonanalytic automorphic
functions}, Nagoya Math. J. \textbf{52} (1973), pages 133-145.


\bibitem{Ono08}
K.~Ono, \emph{Unearthing the visions of a master: harmonic {M}aass forms and
  number theory}, Current Developments in Mathematics \textbf{2008} (2009),
  347--454.

\bibitem{Rankin} R. A. Rankin, \emph{Contributions to the theory of Ramanujan's function $\tau(n)$},
Proc. Camb. Philos. Soc. \textbf{35} (1939), 357--372.

\bibitem{Rankin2} \bysame, \emph{The construction of automorphic forms from the derivatives of a given form}, J. Indian Math. Soc. \textbf{20} (1956), 103--116.

\bibitem{Selberg} A. Selberg,  \emph{Bemerkungen \"uber eine Dirichletsche Reihe, die mit der Theorie der Modulformen nahe verbunden ist}, Arch. Math. Naturvid. \textbf{43} (1940), 47--50.

\bibitem{Selberg2} \bysame, \emph{On the estimation of Fourier coefficients of modular forms},
Proc. Sympos. Pure Math. Vol VIII, Amer. Math. Soc. Providence, RI 1965, 1--15.

\bibitem{St80}
J.~Sturm, \emph{Projections of ${C}^\infty$ automorphic forms}, Bull. Amer.
  Math. Soc. (N.S.) \textbf{2} (1980), no.~3, 435--439.

\bibitem{ZagierRCB} D. Zagier, \emph{Modular forms and differential operators}, 
Proc. Indian Acad. Sci. (Math. Sci.), \textbf{104} (1994), no. 1, 57--75.

\bibitem{ZagierBourbaki} D. Zagier, \emph{Ramanujan's mock theta functions
and their applications
[d'apr\`es Zwegers and Bringmann-Ono]},
S\'eminaire Bourbaki,
60\`eme ann\'ee, 2006-2007, no,. 986.


\bibitem{ZwegersDiss}
S.~Zwegers, \emph{Mock {T}heta {F}unctions}, Ph.D. thesis, Universiteit
  Utrecht, 2002.

\end{thebibliography}
\end{document}